\documentclass{amsart}
\usepackage{amssymb,amsmath,amsfonts,amsthm,amscd,mathrsfs}
 \newtheorem{thm}{Theorem}[section]
 \newtheorem{cor}[thm]{Corollary}
 
 \newtheorem{prop}[thm]{Proposition}
 \theoremstyle{definition}
 \newtheorem{defn}[thm]{Definition}
 \theoremstyle{remark}
 \newtheorem{rem}[thm]{Remark}
 \newtheorem*{ex}{Example}
 \numberwithin{equation}{section}
 
\newcommand{\Cinf}{\ensuremath{\mathcal{C}^\infty}}
\newcommand{\cinfty}{\ensuremath{\mathcal{C}^\infty}}

\newcommand{\D}{\ensuremath{{\mathcal D}}}

\newcommand{\E}{\ensuremath{{\mathcal E}}}

\newcommand{\mb}[1]{\ensuremath{\mathbb{#1}}}

\newcommand{\R}{\mb{R}}

\newcommand{\Z}{\mb{Z}}

\newcommand{\G}{\ensuremath{{\mathcal G}}}

\newcommand{\WF}{\mathrm{WF}}
\newcommand{\singsupp}{\mathrm{sing supp}}
\newcommand{\Char}{\ensuremath{\text{Char}}}

\renewcommand{\d}{\ensuremath{\partial}}

\newfont{\bl}{msbm10 scaled \magstep2}

\newcommand{\beq}{\begin{equation}}
\newcommand{\eeq}{\end{equation}}

\newcommand{\col}{\colon}

\newcommand{\FT}[1]{\widehat{#1}}

\newcommand{\F}{\ensuremath{{\mathcal F}}}

\newcommand{\notmid}{\mid\kern-0.5em\not\kern0.5em}

\newcommand{\norm}[2]{{\| #1 \|}_{#2}}

\newcommand{\al}{\alpha}
\newcommand{\be}{\beta}

\newcommand{\eps}{\varepsilon}

\newcommand{\vphi}{\varphi}

\newcommand{\la}{\lambda}

\newcommand{\sig}{\sigma}

\newcommand{\supp}{\mathop{\mathrm{supp}}}

\newcommand{\sinc}{\ensuremath{\text{ sinc}}}

\newcommand{\smoothing}{\Psi^{-\!\infty}}
\newcommand{\maG}{\mathcal{G}}

\newcommand{\ep}{\varepsilon}
\newcommand{\maD}{\mathcal{D}}

\newcommand{\Sch}{\mathscr{S}}
\newcommand{\CC}{\mathbb C}
\newcommand{\RR}{\mathbb R}
\newcommand{\gs}{{\mathcal G}}

\newcommand{\WFg}{\mathrm{WF}_g}
\newcommand{\maE}{\mathscr{E}}

\begin{document}
%
%
%
%
%
%
%
%
%
\title[]
 {Geometric regularization on Riemannian and \\ Lorentzian manifolds}
\author{Shantanu Dave}

\address{%
Faculty of Mathematics, University of Vienna\\
Nordbergstra{\ss}e 15,
A 1090 Wien\\
Austria}

\email{shantanu.dave@univie.ac.at}

\author{G\"unther H\"ormann}
\address{Faculty of Mathematics, University of Vienna\\
Nordbergstra{\ss}e 15,
A 1090 Wien\\
Austria}
\email{guenther.hoermann@univie.ac.at}


\author{Michael Kunzinger}
\address{Faculty of Mathematics, University of Vienna\\
Nordbergstra{\ss}e 15,
A 1090 Wien\\
Austria}
\email{michael.kunzinger@univie.ac.at}




\begin{abstract}

We investigate regularizations of distributional sections of vector bundles
by means of nets of smooth sections that preserve the main regularity properties
of the original distributions (singular support, wavefront set, Sobolev regularity).
The underlying regularization mechanism is based on functional calculus of elliptic operators with 
finite speed of propagation with respect to a complete Riemannian metric. 
As an application we consider the interplay between the wave equation on a Lorentzian manifold and 
corresponding Riemannian regularizations, and under additional regularity
assumptions we derive bounds on the rate of convergence of their commutator. We  also show that
the restriction to underlying  space-like foliations behaves well with respect to
these regularizations.

\end{abstract}

\maketitle

\section{Introduction}\label{intro}
We consider regularization processes to smooth out distributions on Riemannian and globally hyperbolic Lorentzian manifolds and investigate their compatibility with the wave-equation. 
The Riemannian setting has been addressed in \cite{DHK:10} and we include an introduction
to this approach. For a globally hyperbolic manifold we pick a splitting of the metric as obtained by \cite{BS:05} which provides us with a globally defined time function and a foliation by space-like hypersurfaces.  The associated  Riemannian metric  naturally allows us to construct regularization  processes on distributional sections of tensor-bundles and differential forms. We show that these regularizations interact nicely with the wave-equation on the Lorentzian manifold and with the foliation provided by the metric splitting.

By a regularization process we mean a net of smoothing operators that assigns a net of 
approximating smooth functions (or sections in a vector bundle) to any given distribution
(or distributional section). We are interested in preserving a maximal set of regularity
properties of the distribution in this process (support, singular support, wavefront set,
Sobolev regularity). To assign such properties to approximating nets of smooth objects
we employ the language of algebras of generalized functions in the sense of Colombeau
(\cite{C1,C2,O,GKOS}). In this approach any such regularization process provides an
embedding of distributions into a space $\G(M)$ of nonlinear generalized functions (given
by a quotient construction on spaces of approximating nets). This process
preserves the regularity and singularity structure of the distributions as described in Section \ref{CR} below.   The main interest in employing these regularizations is in studying non-smooth  curved space-times.


Our approach is motivated by work on wave equations on non-smooth curved space-times.  C.J.S.\ Clarke (\cite{Cl1,Cl2,Cl3}) suggested to study physical fields for understanding the singularity
structure of the space-time itself, i.e., to consider wave equations in low regularity. 
His work gave rise to a number of further studies. In particular, in the framework of generalized
functions this line of research was pursued in \cite{VW,GMS:09,SV}. For a more detailed 
introduction we refer to \cite{HKS}.

Thus let us consider a model  situation and  analyze the  generalized solutions to wave equations corresponding to a generalized Lorentz metric $\tilde{g}$ on the smooth manifold $M$ (cf.\ \cite{GKOS,KS:02,GMS:09}), i.e., a global Cauchy problem with initial data $a,b \in \D'(S)$ on a suitable initial value surface $S$ in the form
$$
  \Box_{\tilde{g}}\, u = 0, \quad u\! \mid_S\, = \iota(a),\; 
\nabla_{\text{n}} u\! \mid_S\, = \iota(b),
$$
where $\iota$ is assumed to be an embedding of distributions into the algebra of generalized functions. 

Suppose that the generalized Lorentzian structure induced by $\tilde{g}$ allows a splitting of the wave operator in the form  
$\Box_{\tilde{g}} = \gamma \cdot (\Box_g + Q)$, where $g$ is a smooth ``background'' metric, $\gamma$ is a positive measure or a strictly positive generalized function, and $Q$ is a partial differential operator with distributions or generalized functions as coefficients, but whose coefficient singularities are concentrated in certain space-time regions. A simple example is a generalized Robertson-Walker space-time with $\tilde{g} = - dt^2 + (1 + \mu(t)) h$ on $M = \R \times S$, where $h$ is a smooth Riemannian metric on the smooth manifold $S$ and $\mu$ is a nonnegative generalized function. In this case $g = -dt^2 + h$ can serve as a smooth background metric and the above splitting would involve $\gamma = 1/(1 + \mu)$ and $Q = Q(t,\d_t)$, which acts only on the one-dimensional factor.

If we impose the splitting assumption as above, then we have for any generalized function $u \in \G(M)$ that the equation $\Box_{\tilde{g}}\, u = 0$ is equivalent to $\Box_g\, u + Q u = 0$. Let $u \in \G(M)$ be a solution to the Cauchy problem stated above. In attempting to extract its distributional aspects or assess its singularity structure we could adopt the following strategy. Let $v$ be the distributional solution (assuming that it exists and is unique) to
$$
  \Box_g\, v = 0, \qquad  v\! \mid_S\, = a, \quad
\nabla_{\text{n}} v\! \mid_S\, = b,
$$
and put $w := \iota(v) - u \in \G(M)$ as a comparison of the generalized function $u$ with its ``background distributional aspect''. Now, \emph{if the embedding $\iota$ commutes with restriction to $S$ and with $\Box_g$}, then simple manipulations allow us to draw the conclusion that $w$ satisfies the following Cauchy problem:
$$
   \Box_g w = Q u, \qquad  w\! \mid_S\, = 0,\quad 
  \nabla_{\text{n}} w\! \mid_S\, = 0. 
$$
Thus, $w$ satisfies an inhomogeneous Cauchy problem corresponding to the smooth background metric $g$ and with generalized functions occurring only on the right-hand side of the equation. In particular, $w$ will be smooth in regions where null geodesics corresponding to $g$ emanating from $S$ do not intersect the $\Cinf$-singular support of $Q u$.

In the more general situation of a globally hyperbolic Lorentz manifold  provided with a foliation
into spacelike hypersurfaces via a splitting of the metric we shall construct  an embedding $\iota $  and show that the embedding almost commutes with $\Box_g$ and with restriction to space-like slices of the foliation.  The extent of failure to commute is roughly  measured by the divergence of the volume element on the slices along the time  direction.  Although we state our results in the scalar case, the extension to the case of  wave-equations on differential forms is obvious.

\section{Regularization of distributions on complete Riemannian manifolds}\label{CR}

In this section we describe a geometric regularization procedure for distributions  (or distributional sections of vector bundles) on complete Riemannian manifolds. This procedure encodes regularity and singularity features in terms of asymptotic behavior. Our approach is based on  \cite{DHK:10}, to which we refer for further details.

We first fix some notations from the theory of distributions.
Let $M$ be an orientable complete Riemannian manifold of dimension
$n$ with Riemannian metric $g$. The space $\D'(M)$ of Schwartz distributions on $M$
is the dual of the space $\Omega_c^n(M)$ of compactly supported $n$-forms
on $M$. Further, $\D(M)$ is the space of smooth compactly supported functions
on $M$. We identify $\D(M)$ with
$\Omega_c^n(M)$ via $f\mapsto f\cdot dg$, with $dg$ 
the Riemannian volume form induced by $g$. Thus $\D'(M)$ can be viewed as
the dual space of $\D(M)$. We consider $L^1_{\mathrm{loc}}(M)$ (hence in particular 
$\cinfty(M)$) a subspace of $M$ via $f\mapsto [\vphi \mapsto \int_M f\vphi dg]$.
If $E$ is a vector bundle over $M$ then $\D'(M:E)$, the space
of $E$-valued distributions on $M$ is defined by $\D'(M:E) = \D'(M)\otimes_{\cinfty(M)} 
\Gamma^\infty(M:E)$, with $\Gamma^\infty(M:E)$ the space of smooth sections
of $E$. By ${\mathcal E}'(M:E)$ we denote the space of compactly supported distributional
sections of $E$.
We shall assume that $E$ is endowed  with a Hermitian inner product so that the distributional sections $\maD'(M:E)$ of $E$ can be identified with the dual $\Gamma_c^{\infty}(M:E)'$ of the space of compactly supported smooth sections. 
Regularity properties of distributions we are interested in are, in particular,
encoded in the singular support $\singsupp(w)$, the wavefront set $\WF(w)$, and
the Sobolev regularity of any given $w\in \D'(M)$.


In order to be able to track these regularity properties in terms of regularizations of a
given distribution we need a conceptual framework that allows to assign geometrical and
analytical properties to regularizations, i.e., to nets of smooth functions. Algebras of generalized
functions in the sense of J.\ F.\ Colombeau (\cite{C1,C2,O}) 
indeed provide a well-developed theory of this
kind and below we shall briefly review some basic definitions, based mainly on \cite{GKOS,Gtop}.

The basic idea of Colombeau's approach is to assign to any given locally convex space $X$
a space $\G_X$ of generalized functions as follows. We define the space ${\mathcal M}_X$
of moderate nets to consist of those maps $\eps\mapsto x_\eps$ ($\eps \in I:=(0,1]$) that are smooth
and for any seminorm $\rho$ of $X$ satisfy $\rho(X_{\varepsilon}) = O(\varepsilon^N)$ 
for some integer $N$ as
$\eps\to 0$. Similarly we call a net  $x_{\varepsilon}$ negligible if for all seminorms $\rho$ and all integers $N$ the asymptotic relation $\rho(X_{\varepsilon})\sim O(\varepsilon^N)$ holds. The space $\maG_X$ is defined as the quotient of the space of all moderate nets by the space 
${\mathcal N}_X$ of all negligible nets. The class represented by a net $x_{\ep}$ shall be denoted by $[x_{\ep}]$.  In case $M$ is a smooth manifold we call $\maG(M):=\maG_{\mathcal{C}^{\infty}(M)}$
the standard (special) Colombeau algebra
of generalized functions on $M$ (\cite{C1,DD,GKOS}).
If $E\rightarrow M$ is a vector bundle then we set $\maG(M:E):=\maG_{\Gamma^{\infty}(M:E)}$. 
For $E=\CC$ the space $\maG_{\CC}$ inherits a  ring structure from $\CC$. It is therefore
called the space of generalized numbers, denoted by $\tilde{\CC}$. 
Every space $\maG_E$ is naturally a $\tilde{\CC}$-module, and hence  is 
called the $\tilde{\CC}$-module associated with $E$ (\cite{Gtop}).

Similarly we shall also consider the subspace $\maG^{\infty}_X$ of regular elements of $\G_X$, defined as the space of all those elements of $\maG_X$ that can be represented by a net $x_{\varepsilon}$ for which one can find an integer $N$ in the above  relations independent of any seminorm $\rho$  on $X$. Thus $\maG^{\infty}_X$  is the space of uniformly controlled asymptotics.  Again we set $\maG^{\infty}(M):=\maG^{\infty}_{\mathcal{C}^{\infty}(M)}$. In regularity theory,
$\G^\infty(M)$ is the analogue of $\cinfty(M)$ in the theory of distributions (\cite{O,H,DPS,GH}).
This is based on the fundamental result (\cite{O}) that for open subsets $\Omega$ of $\R^n$,
$\iota(\D'(\Omega))\cap \G^\infty(\Omega) = \cinfty(\Omega)$, where $\iota$ is the standard
embedding of $\D'$ into $\G$ via convolution with a standard mollifier. We will introduce
further notions of regularity theory based on $\G^\infty$ below. 


The assignment $X\mapsto  \maG_X$ is obviously functorial, so that  any continuous linear map $\phi:X\rightarrow Y$  naturally induces a map $\phi_*:\maG_X\rightarrow \maG_Y$. Thus a smooth map $f:M\rightarrow N$  gives rise to a pull-back $f^*:\maG(M)\rightarrow \maG(M)$. In particular $\maG(M)$ defines a fine sheaf of algebras and similarly $\maG(M:E)$ is a fine sheaf of $\maG(M)$-modules.

We finally introduce the notion of association, which provides a concept of weak equality
between elements of  Colombeau spaces as well as a standardized way of assigning distributional
limits to certain Colombeau generalized functions. For $u$, $v\in \G(M:E)$ we say that
$u$ is associated (or weakly equal) to $v$, $u\approx v$ if $u_\eps - v_\eps\to 0$ in $\D'(M:E)$ for (some, hence any) 
representatives of $u$, $v$. We say that $u$ possesses $w\in \D'(M:E)$ as a distributional
shadow if $u_\eps \to w$ in $\D'(M:E)$.

To construct the desired regularization process on $\maD'(M:E)$ we need two pieces of data:
\begin{enumerate}
 \item A Schwartz function $F\in \Sch(\R)$ such that $F\equiv 1$ near the origin.
 \item An elliptic symmetric  differential operator $D$ on $E$ such that the speed of propagation  $C_D$, defined as
\begin{equation}\label{cd}
C_D=\operatorname{sup} \{\|\sigma_D(x,\xi)\|~|~ x\in M,\ \|\xi\|=1\}
\end{equation}
is finite.
Here $\sigma_D$ denotes the principal symbol of $D$.
Such an operator $D$ will be referred to as {\em admissible}.
\end{enumerate}
 As a consequence of the finite speed of propagation, the symmetric operator $D$ is essentially self-adjoint (\cite[10.2.11]{HR}).  Therefore the equation
 \begin{eqnarray}\label{transport_eqn}
 \frac{\partial}{\partial t}u=iDu\qquad u(\,.\,,0)=u_0,
 \end{eqnarray}
 has a unique solution for all times $t$ for any initial datum $u_0\in\Gamma_c^{\infty}(M:E)$.
 Indeed it follows from functional calculus that $e^{itD}u_0$ is a solution, and uniqueness
 can be established using energy estimates.
 
We note that functional calculus defines a map 
$$
\Sch(\RR)\ni f\mapsto f(D) \in \mathcal{B}(L^2(M:E))
$$ 
where the operator $f(D)$ can also be expressed by the Fourier inversion formula:
\begin{eqnarray}\label{FIF}
f(D):=\frac{1}{2\pi}\int_{\RR}\hat{f}(s)e^{isD}ds
\end{eqnarray}
(in the sense of strong operator convergence).
 Since the operator $D$ is elliptic, it  follows that $f(D)$  has a smooth kernel by elliptic regularity.

 Let $F\in\Sch(\RR)$   with $F\equiv 1$ near the origin and set $F_{\varepsilon}(x):=F(\varepsilon x)$. Then we shall obtain the desired regularizing procedure  from $F_{\varepsilon}(D)$. To do this we first take a closer look at  singularity properties available in generalized functions.
 
As has been mentioned above the space $\maG(M:E)$ is a fine sheaf over $M$, hence provides a notion of support. Furthermore,
the notion of wave-front set can be defined for generalized functions both locally and in invariant global terms analogous to  distributions.
For $\Omega\subseteq \R^n$ open we call $u\in \gs(\Omega)$
$\gs^\infty$-microlocally regular at $(x_0,\xi_0)\in T^*\Omega\setminus 0$ if there exists some
test function $\varphi
\in \D(\Omega)$ with $\varphi(x_0)=1$ and a conic neighborhood $\Gamma\subseteq \R^n\setminus 0$ of $\xi_0$ such that
the Fourier transform ${\mathcal F}(\varphi u)$ is rapidly decreasing in $\Gamma$, i.e., there exists $N$ such that for all $l$,
$$
 \sup_{\xi\in \Gamma} (1+|\xi|)^l|(\varphi u_\eps)^\wedge(\xi)| = O(\eps^{-N}) \qquad (\eps\to 0).
$$
The generalized wave front set of $u$, $\WFg(u)$, is the complement of the set of points $(x_0,\xi_0)$
where $u$ is $\gs^\infty$-microlocally regular. By \cite{Ha}, for any $u\in \maG(M)$, $\WFg(u)$ can naturally be viewed as a subset of $T^*M\setminus0$ . It is  equivalently defined as  (cf.\cite{GH}),
\[
\WFg(u)=\bigcap_{Pu\in\maG^{\infty}(M)} \operatorname{char}(P)\qquad (P\in\Psi^0_{cl}(M)).
\]
Here $\operatorname{char}(P)\subseteq T^*M$ is the characteristic set of the
order  $0$ classical pseudodifferential operator $P$. The singular support of $u$, $\singsupp_g(u)$, is the complement of the maximal open set on which $u\in \G^\infty$. It then follows that $\singsupp_g(u) = \mathrm{pr}_1(\WFg)$.


We are now ready to define  a class of regularization procedures 
for distributional sections of a vector bundle.
\subsection{Regularizations}
\begin{defn}\label{bundle_embedding}
   A parametrized family $(T_{\varepsilon})_{\eps\in I}$ of properly supported smoothing operators (in the sense
   of \cite[ch.\ 1.4]{CP}) is
 called an optimal regularization process if
 \begin{enumerate} 
 \item \label{moderatevb}
 The regularization of any compactly supported distributional section $s\in\mathcal{E}'(M:E)$ is of moderate growth:
For any continuous seminorm $\rho$ on $\Gamma^{\infty}(M:E)$, there exists some integer $N$ such that
   \[
\rho(T_{\varepsilon}s) = O(\varepsilon^{N}) \qquad (\eps\to 0).
\]
\item\label{Identityvb} The net $(T_{\varepsilon})$ is an approximate identity: for each
$s\in \maE'(M:E)$, 

\[\underset{\varepsilon \rightarrow 0}{\operatorname{lim}}T_{\varepsilon}s=s\quad \textrm{in}\,\,\maD'(M:E).\]

\item\label{negligiblevb} If $u\in\Gamma_c^{\infty}(M:E)$ is a smooth compactly supported section of $E$ then
for all continuous seminorms $\rho$ and given any integer $m$,
\[
\rho(T_{\varepsilon}u-u) = O(\varepsilon^m).
\]

\item\label{singularityvb} 
The  induced map $\iota_T: {\mathcal E}'(M:E)\rightarrow  \maG(M:E)$  preserves  support, and extends to a sheaf map $\D'(M:E)\rightarrow \maG(M:E)$ that satisfies, 
\begin{eqnarray}\label{MO}
\iota_T(\D'(M:E))\cap \maG^{\infty}(M:E)= \Gamma^{\infty}(M:E). 
\end{eqnarray}
This implies in particular that $\iota_T$ preserves singular support.
\item The map $\iota_T$ preserves wave-front sets in the sense that for any distribution $s\in\D'(M:E)$
\[\WF(u)=\WF_g(\iota_T(u)).\]
\end{enumerate}
\end{defn}
As mentioned already we shall use spectral properties of the  elliptic differential operators  
to obtain  regularizing processes satisfying the  above conditions.

Our main result in this section is the following.

\begin{thm}\label{operator_embedding}
Let  $F\in \Sch(\RR)$ be a  Schwartz function such that $F\equiv 1$ near the origin. Let $F_{\varepsilon}(x):=F(\varepsilon x)$. 
Given an admissible differential operator $D$, 
the family of operators $(F_{\varepsilon}(D))_{\eps\in I}$ provides an optimal regularization process in sense of Def.\ \ref{bundle_embedding}.
\end{thm}
The essential idea of the proof boils down to the following two steps.

First we prove  Theorem \ref{operator_embedding} under the assumption that  the underlying  manifold $M$ is compact. We then use finite propagation speed to extend the result to a general complete Riemannian manifold.

  Thus let us for the time being assume that $M$ is a closed  manifold. Then the space of smoothing operators $\Psi^{-\!\infty}(M:E)$ is a Frech\'et algebra. 
 The functional calculus map defined by $D$  induces a smooth map $\phi_D:\Sch(\RR)\rightarrow \smoothing(M:E)$, hence the estimate \ref{moderatevb} holds because $F_{\ep}$ is a moderate net and $F_{\ep}(D)=(\phi_D)_*(F_{\ep})$, where $(\phi_D)_*$ is the induced map on the asymptotic spaces $\maG_{\Sch(\RR)}\rightarrow \maG_{\smoothing(M)}$. Condition \ref{Identityvb} then follows as the ring map $\phi_D$ preserves approximate units.

When $M$ is compact the operator $D$ has discrete spectrum and the spectrum of $D^2$ satisfies  Weyl's law, namely
\begin{eqnarray}\label{weyl}
N_{D^2}(\lambda):=\#\{\lambda_i\in \operatorname{sp}(D^2)|~~\lambda_i\leq \lambda\}\sim C
\lambda^{\frac{\operatorname{dim}(M)}{2}}.
\end{eqnarray}
The Weyl estimates in conjunction with the fact that $F\equiv 1$ near the  origin now provides the compatibility condition \ref{negligiblevb} (see \cite{D}).

To describe the regularity of a distributional section  $u\in\D'(M:E)$ we first note that any distribution provides a map between two Frech\'et spaces, namely between the smoothing operators  $\smoothing (M:E)$ and the smooth sections  $\Gamma^{\infty}(M:E)$ by evaluation. More precisely, to any $u\in\maD'(M:E)$ we associate the map
\[
\Theta_u: \smoothing(M) \rightarrow\Gamma^{\infty}(M:E)\qquad
\Theta_u(T):=T(u).
\]
 The mapping properties of the maps $\Theta_u$ imply that if $u\not\in H^k(M:E)$ for every $k>t$  then given any $\delta>0$,
$\|F_{\varepsilon}(D)u\|^2_{L^2(M:E)}$ is not $O(\varepsilon^{\frac{\operatorname{dim~M}}{2}+t+\delta})$ (see \cite[Lemma 7.4]{D}).
This direct description of Sobolev regularity is in fact stronger than condition 4 and therefore,
\[\iota_{F_{\varepsilon}(D)}(\D'(M:E))\cap \maG^{\infty}(M:E)= \Gamma^{\infty}(M:E)).\]

The sheaf property of the embedding $\iota_{F_{\ep}(D)}$ as well as the proof of the result for the more general case of a not necessarily compact complete manifold  rely on the finite speed of propagation of $D$, to which we turn next.
\subsection{Finite speed of propagation}	
We shall reduce the computation of asymptotics on a complete Riemannian manifold to certain compact manifolds obtained as doubles of suitable compact submanifolds with boundaries.
If $X$ is a compact manifold with boundary,  a double of $X$, 
denoted here by $DX$ is a closed manifold obtained by  gluing two copies of $X$  along the boundary $\partial X$ (see, e.g., \cite{Ko}).   
If $X$ is a compact manifold with boundary embedded in a Riemannian manifold $M$ of the same dimension and 
if $U$ is an open subset of $M$ such that $\bar{U}\subset \operatorname{interior}(X)$, then 
there exists a 
Riemannian metric on $DX$ such that the inclusion $j:U\hookrightarrow  DX$ becomes an isometric embedding.  Moreover, for any vector bundle $E\rightarrow M$ there exists a corresponding
vector bundle $E_X\rightarrow DX$ such that $E_X|_U$ 
is canonically isomorphic to $E_{|_U}$. Also, there exists a symmetric  elliptic operator 
$D_X$ on $E_X$ that coincides with $D$ on $U$.

Let $u\in\mathcal{E}'(M:E)$ and fix a constant $c>0$.  Then by the Hopf Rinow theorem 
the open ball $U:=B_{2c\cdot C_D}(\operatorname{supp}(u))$ (with $C_D$ as defined in (\ref{cd}))
is relatively compact and therefore 
contained in a compact manifold  
with boundary $X\subseteq M$. By the above $u$ can be identified with a distributional section of $E_X\rightarrow  DX$.

\begin{prop}\label{sheaf} With assumptions  on $u,c$ and $F$ as above $F(D)u$ and $F(D_X)u$ are both supported in $U$ and 
\[F(D)u=F(D_X)u.\]
\end{prop}
\begin{proof}
The restrictions of $D$  and $D_X$  to the open set $U$ coincide, hence uniqueness of solutions to \eqref{transport_eqn}  implies that $e^{isD}u$ and $e^{isD_X}u$   agree for $s\le c$.  Hence the
claim follows from the Fourier Inversion Formula  \eqref{FIF}.
\end{proof}
From these observations the proof of Th.\ \ref{operator_embedding} can be deduced along 
the following lines (cf.\ \cite[Sec.\ 4]{DHK:10}):

Let $\phi\in \D(-c,c)$ satisfy $\phi\equiv 1$ in a neighborhood of $0$ and let $u\in \E'(M:E)$. Then
$$
 F_{\varepsilon}(D)u =\frac{1}{2\pi}\int^{\infty}_{-\infty}\phi(s)\hat{F_\eps}(s)e^{isD}u\,ds
  =j^*\left(\frac{1}{2\pi}\int^{\infty}_{-\infty}\phi(s)\hat{F_\eps}(s)e^{isD_X}u\,ds\right)
$$
 
Based on this calculation we observe that the estimates  required for Definition \ref{bundle_embedding} in the general case of $F_{\ep}(D)$ follow from that of $F_{\ep}(D_X)$ which have already been established in case of the closed manifold $DX$. Furthermore the sheaf properties of $\iota_{F_{\ep}(D)}$ can be obtained as a consequence of the calculations in Proposition \ref{sheaf}.   
     The support of $u$ coincides with the generalized support of $[F_{\ep}(D)u]$.  This implies that the embedding extends  to a sheaf morphism $i_{F_{\varepsilon}}:\D'(M:E)\rightarrow  \maG(M:E)$.
Preservation of wavefront sets is more involved. We refer to \cite[Th.\ 3.10]{DHK:10} for a
complete proof of this property.  
\begin{rem}
In the scalar case, an alternative proof (not relying on the above doubling-technique) of Theorem \ref{operator_embedding} can be found in \cite{DHK:10}, Section 3.
\end{rem}

\subsection{Applications} As a first application consider $E=\bigwedge^*M$, the exterior bundle over a Riemannian manifold $M$ and $D=d+d^*$ where $d^*$ is the Hodge adjoint of the de Rham differential $d$. Then $D$ is symmetric, elliptic and has propagation speed $C_D=1$ since $\sigma_D(x,\xi)^2= -\|\xi\|^2\mathrm{id}$, hence is essentially self-adjoint. Letting $\Delta:= D^2$ be the
Laplace-Beltrami operator on $\Omega^*(M)$, it follows from functional calculus that
$\cos(sD) = \cos(s\sqrt{\Delta})$ on $L^2(M)$.

Now let $F$ be an even Schwartz function which has germ $1$ at the origin. Then
$$
F(\sqrt{\Delta})=\frac{1}{2\pi}\int_{-\infty}^{\infty} \hat{F}(s)\cos(s\sqrt{\Delta})\,ds
$$ 
as a Bochner integral in $\mathcal{B}(L^2(M))$.

Finally, let $c>0$ and pick an even test function $\phi_c$ with support in $(-2c,2c)$ and
such that $\phi_c \equiv 1$ on $(-c,c)$. Now set
\begin{equation} \label{tepsdef}
T_\eps(\sqrt{\Delta}) := \frac{1}{2\pi} \int_{-\infty}^\infty \phi_c(s) (F_\eps)^\wedge(s) \cos(s\sqrt{\Delta})\, ds.
\end{equation}
Then each $T_\eps(\sqrt{\Delta})$ is a properly supported smoothing operator and
$(T_\eps(\sqrt{\Delta}))_{\eps\in I}$ is an optimal regularization process in the sense of
Def.\ \ref{bundle_embedding}.
From the construction of $T$ in terms of the functional calculus of the Laplace operator we
conclude the following invariance properties of the corresponding embedding $\iota_T$:
\begin{cor} \
\begin{itemize} 
\item[(i)] Let $f:M\to M$ be an isometry. Then for any $u\in \D'(M)$, $\iota_T(f^*u) = f^*\iota_T(u)$.
\item[(ii)] If $\Psi$ is a pseudodifferential operator commuting with $\Delta$, then $\Psi$
commutes with $\iota_T$.
\end{itemize}
\end{cor}

For the special case $\R^n$ with the Euclidean metric the above construction gives 
$$
  T_\eps u = \mu_\eps * u, \quad \text{with } \quad 
    \FT{\mu_\eps}(\xi) = \frac{1}{2\pi} \F\big(\phi \frac{1}{\eps} 
      \FT{F}(\frac{.}{\eps})\big)(|\xi|).
$$
In particular, for the one-dimensional case $n=1$ we obtain $\mu_\eps = \phi_c \cdot \FT{F}(./\eps)/(2 \pi \eps)$. Note that $\FT{F}(./\eps)/(2 \pi \eps)$ is a standard mollifier, i.e., a Schwartz function with unit integral and all higher moments vanishing. Thus the regularization
process reduces precisely to the usual Colombeau embedding via convolution (\cite{C2,O,GKOS}).

 A further  consequence of the above construction is that  it relates naturally to isomorphisms of vector bundles. Thus let $\phi:E_1\rightarrow E_2$ be a vector-bundle isomorphism (covering the identity map on $M$) and let $D_1$ be an order  one admissible operator on $E_1$. We  choose a Hermitian structure on $E_2$ which makes $\phi$ an isometry. Then the push-forward operator $D_2:=\phi D_1\phi^{-\!1}$ is isospectral to $D_1$.  The naturality of functional calculus then gives that for any Schwartz function $F$ with $F\equiv 1$ near the origin the embedding and the bundle map $\phi$  commute, that is 
 \[\phi\circ F_{\ep}(D_1)=F_{\ep}(D_2)\circ\phi.\]
 We note that if $r_1+s_1=r_2+s_2$ then any Riemannian metric provides an isomorphism  of the tensor bundles $\mathfrak{T}^{r_1}_{s_1}\rightarrow \mathfrak{T}^{r_2}_{s_2}$  by `raising  or lowering of indices'. Thus if  we pick the  connection Laplace operators on the tensor bundles and an even Schwartz function $F$ as above we obtain a regularization process that is well behaved with respect to raising and lowering of indices.

\section{Regularization on globally hyperbolic space-times}

We now return to the situation described at the end of Section \ref{intro}.
Thus, throughout this section, $(M,g)$ will be a smooth space-time, i.e., a connected time-oriented Lorentz manifold. We first review the concept of metric splitting for globally hyperbolic 
space-times. Building on this we construct invariant regularizations of distributions 
on smooth globally hyperbolic space-times.

\subsection{The metric splitting of globally hyperbolic space-times and associated Riemann metrics}
The original definition, due to J. Leray of global hyperbolicity of a space-time $M$
appeared in \cite{L53}. It requires that the set of causal curves connecting two points $p$,
$q\in M$ be compact in the space of all rectifiable paths with respect to a suitable metric topology (\cite[Ch.\ XII, Sec.\ 8, 9]{CB}). In what follows, we will use an equivalent definition,
cf.\ \cite[Sec.\ 6.6]{HE}, and \cite[Ch.\ XII, Th.\ 10.2]{CB}.

 Thus we call a space-time \emph{globally hyperbolic} if it satisfies (a) strong causality and (b) for any $p,q \in M$ the intersection $J^+(p) \cap J^-(q)$ of the causal future $J^+(p)$ of $p$ with the causal past $J^-(q)$ of $q$ is a compact subset of $M$. Thanks to \cite{BS:07} condition (a) may be weakened to causality, i.e., non-existence of closed time-like curves.

As shown by Geroch (cf.\ \cite{Geroch:70}) global hyperbolicity is equivalent to the existence of a Cauchy hypersurface and in turn provides a foliation of $M$ by Cauchy hypersurfaces. Further development of these constructions and techniques led to the following result due to Bernal-S\'{a}nchez (cf.\ \cite{BS:05}) on the so-called metric splitting of a globally hyperbolic space-time $(M,g)$: 
There exists a Cauchy hypersurface $S$ in $M$ and an isometry of $(M,g)$ with the Lorentz manifold $(\R\times S, \la)$ with Lorentz metric $\la$ given by
\beq\label{lambda}
    \la = - \be\, dt^2 + h_t,
\eeq
where $\be \in \Cinf(\R\times S)$ is positive, $(h_t)_{t\in\R}$ is a smoothly parametrized family of Riemannian metrics on $S$, and $t$ denotes (slightly ambiguously) both the global time function $(t,x) \mapsto t$ and its values. In other words, in order to construct a regularization (or embedding) for distributions on the globally hyperbolic Lorentzian manifold $(M,g)$ we may as well 
assume that $(M,g) = (\R\times S, \la)$. 

The specific structure of the Lorentz metric $\la$ in \eqref{lambda}  suggests to associate with it the Riemann metric 
\beq\label{rho}
  \rho := \be\, dt^2 + h_t
\eeq  
on $\R \times S$ and to simply use the regularization and embedding of distributions on $\R\times S$ based on this Riemann metric $\rho$. The regularization construction on Riemannian manifolds described above requires completeness of the Riemann metric $\rho$. It may happen that $\rho$ is not complete, however we may then introduce an appropriate conformal factor to obtain a complete Riemannian metric (cf.\ \cite{NO:61}). Note that the latter would amount to introducing the exact same conformal factor for the Lorentz metric $\la$ and would not change the class in the so-called causal hierarchy of space-times according to Minguzzi-S\'{a}nchez \cite{MS:08}. Therefore we assume henceforth that $\rho$ is complete.

Denoting by $\Delta_\rho$ the Laplacian w.r.t.\ $\rho$, for any $s\in \R$ and any $u\in \D(M)$ we set
$$
\|u\|_s := \| (1+\Delta_\rho)^{s/2} u \|_{L^2(M)}.
$$
The Sobolev space $H^s(M)$ of order $s$ is the completion of $\D(M)$ with respect to this norm. 

\begin{ex} (i) Consider $S = \R$ and $\la = - dt^2 + \frac{dx^2}{1 + t^6 x^6}$ as Lorentz metric on $\R^2$. It is elementary to check that $(\R^2,\la)$ is globally hyperbolic, e.g., by showing that $\{0\} \times \R$ is a Cauchy hypersurface. The associated Riemann metric $\rho =  dt^2 + \frac{dx^2}{1 + t^6 x^6}$ is not complete, since the hypersurfaces $\{t\} \times \R$ with $t \neq 0$ are closed and bounded but not compact in $(\R^2,\rho)$. Multiplying $\rho$ by the function $\alpha \in \Cinf(\R^2)$, $\alpha(t,x) = 1 + t^6 x^6$, yields a conformal metric which is complete (by the Hopf-Rinow theorem, since lengths of curves w.r.t.\ $\alpha\cdot \rho$ are greater than or equal to their Euclidean lengths).

(ii) Let $(S,h_0)$ be a connected Riemannian manifold. For the Robertson-Walker space-time $\R \times S$ with  Lorentz metric of the form $\la = -dt^2 + f(t)^2 h_0$, where $f \in \Cinf(\R)$ is positive, we have (cf.\ \cite[Lemma A.5.14]{BGP:07}): 

\centerline{$(\R\times S,\la)$ is globally hyperbolic if and only if $(S,h_0)$ is complete.} 

\noindent If this is the case, then the corresponding Riemannian metric $\rho = dt^2 + f(t)^2 h_0$ on $\R\times S$ is complete (see \cite{ON:83}, Lemma 7.40). 

\end{ex}

\subsection{Regularization and embedding via the associated Riemannian structure}

Let $(T_\eps)_{\eps \in I}$ be the regularization (i.e., family of properly supported smoothing operators) and $\iota \col \D'(\R\times S) \hookrightarrow \G(\R\times S)$ be the embedding associated with the complete Riemann metric $\rho$ as in (\ref{tepsdef}). We recall that the embedding  thus respects 
the differential algebraic structure of $\Cinf(\R\times S)$ and the wave front sets of distributions in the strong sense of generalized functions, that is, with respect to equality on the level of $\G(\R\times S)$. Furthermore, $\iota$ is also invariant under isometries of the Riemannian structure and commutes with the action of the Laplace operator $\Delta_\rho$ corresponding to $\rho$ on distributions and generalized functions, respectively.

Let $u \in \D'(\R\times S)$. Then according to \eqref{lambda}
the difference between the wave operator $\Box_\lambda$ and the Laplace operator $\Delta_\rho$ acts on $u$ as
\beq\label{Theta}
    \Box_\la u - \Delta_\rho u = 
   \frac{-2}{\sqrt{\be \det h_t}} \,\d_t 
   \Big( \sqrt{\frac{\det h_t}{\be}} \,\d_t u \Big) := - 2 \Theta u. 
\eeq

From this equation and the commutation property of the embedding $\iota$ with $\Delta_\rho$ we obtain for any $u \in \D'(\R\times S)$
\beq\label{commutators}
  \iota(\Box_\la u) - \Box_\la \iota(u) = 
  \iota((\Delta_\rho - 2 \Theta)u) - (\Delta_\rho - 2 \Theta) \iota(u) = 
  2 \big(\Theta \iota(u) - \iota(\Theta u)\big).
\eeq
Thus, the precise invariance properties of the embedding with respect to the wave operator can be reduced to investigating the corresponding behavior upon interchanging $\iota$ with $\Theta$. 

As a first simple, but useful, observation we point out that $\iota$ commutes with $\Box_\la$ in the distributional sense (i.e., in the sense of association in $\G$), which follows from the fact that $\Box_\la$ is a differential operator with smooth symbol. This proves the following statement.

\begin{prop} For any $u\in \D'(\R\times S)$ we have $\iota(\Box_\la u) \approx \Box_\la\, \iota(u)$.
\end{prop}

We will now show that under additional regularity assumptions on the distribution $u$ a stronger asymptotic property holds.

\begin{thm} \label{asymptth} If $u \in H^{3}(\R \times S)$ is compactly supported then 
$$
\norm{\,[T_\eps, \Box_\la] u\,}{L^2} = 
\mathcal{O}(\eps^2)  \qquad (\eps \to 0).
$$ 
\end{thm}

\begin{proof} 
As in \eqref{commutators} we have the basic relation
$$
    [T_\eps,\Box_\la] u = 2 [\Theta, T_\eps] u.
$$
Now recalling the Bochner-integral defining the action of $T_\eps$ on $L^2$-functions we have
$$
  \Theta T_\eps u - T_\eps \Theta u = \frac{1}{2 \pi} \int_\R 
    \phi_c(s) \frac{1}{\eps} \FT{F}(\frac{s}{\eps}) 
    \Big( \Theta \cos(s \sqrt{- \Delta_\rho}) u - 
    \cos(s \sqrt{- \Delta_\rho}) \Theta u \Big)\, ds.
$$ 
Hence everything boils down to deriving asymptotic estimates for the commutator of $\Theta$ with the operator $S(s) := \cos(s \sqrt{- \Delta_\rho})$ ($s \in \R$). For any $v\in L^2(\R\times S)$ the function $w \in \Cinf(\R, L^2(\R\times S))$, $w(s) := S(s) v$ ($s \in \R$), is the mild solution to the following Cauchy problem on $\R \times \R \times S$:
\beq\tag{$*$}
   \d^2_s w - \Delta_\rho w = 0, \quad w(0) = v,\quad \d_s w(0) = 0.
\eeq
In the above integral formula the term $\cos(s \sqrt{- \Delta_\rho}) \Theta u$ corresponds to the solution with $v = \Theta u$, whereas the term $\Theta \cos(s \sqrt{- \Delta_\rho}) u$ is just the 
application of $\Theta$ 
to the solution $w(s)$ corresponding to $v = u$. Applying $\Theta$ to ($*$) we obtain 
$$
  \d_s^2 \Theta w - \Delta_\rho \Theta w = [\Theta, \Delta_\rho] w =: f
$$
and
$$
 \Theta w \mid_{s = 0} = \Theta(w(0)) = \Theta u, 
  \quad \d_s \Theta w \mid_{s=0} = \Theta (\d_s w(0)) = 0. 
$$
Therefore the Duhamel principle yields
\begin{multline*}
   \Theta \cos(s \sqrt{- \Delta_\rho}) u = 
   \Theta w(s)\\ =  \cos(s \sqrt{- \Delta_\rho}) \Theta u +
   \int_0^s (s-r) \sinc\big((s-r)\sqrt{-\Delta_\rho}\,\big) 
   f(r)\, dr,
\end{multline*}
where $\sinc \col \R \to \R$ is given by $\sinc(z) = \frac{\sin(z)}{z}$. In summary, we obtain
$$
   \Theta T_\eps u - T_\eps \Theta u = \frac{1}{2 \pi} \int_\R 
    \phi_c(s) \frac{1}{\eps} \FT{F}(\frac{s}{\eps}) 
    \int_0^s (s-r) \sinc\big((s-r)\sqrt{-\Delta_\rho}\,\big) 
   f(r)\, dr\, ds,
$$
where 
$$
   f(r) =  [\Theta, \Delta_\rho] w(r) = [\Theta, \Delta_\rho] 
   \cos(r \sqrt{- \Delta_\rho}) u.
$$
Since $[\Theta, \Delta_\rho]$ is of third order and maps $H^3_{\mathrm{comp}}$ into $L^2$ and
$\cos(r \sqrt{- \Delta_\rho})$ has operator norm $1$ on every Sobolev space there exists a constant $C_1 > 0$ such that $\norm{f(r)}{L^2} \leq C_1 \norm{u}{H^3}$ for every $r \in \R$. Furthermore, the operator norm of $(s-r) \sinc\big((s-r)\sqrt{-\Delta_\rho}\,\big)$ is bounded by $\sup_{z\in\R} |(s-r) \sinc((s-r)z)| = |s-r|$. Combining these upper bounds we estimate 
\begin{multline*}
  \norm{\Theta T_\eps u - T_\eps \Theta u}{L^2} \leq
  \frac{C_1}{2 \pi} \int_\R 
    |\phi_c(s)| \frac{1}{\eps} |\FT{F}(\frac{s}{\eps})| 
    \frac{s^2}{2}\, ds \cdot \norm{u}{H^3}\\ 
    = \frac{C_1}{4 \pi} \norm{u}{H^3} 
    \int_\R |\phi_c(\eps \sig)| |\FT{F}(\sig)|
    \sig^2 \, d\sig \cdot \eps^2
    \leq \frac{C_1}{4 \pi} \, \norm{u}{H^3} \norm{\phi_c}{L^\infty} 
    \norm{\FT{F''}}{L^1}\cdot \eps^2.
\end{multline*}
\end{proof}

\begin{rem}\label{comm_lemma} Applying the reasoning of the proof of Theorem \ref{asymptth} 
to the operators $\d_t$ or $M_\alpha$ (multiplication by $\alpha$) instead of $\Theta$ gives
the following additional asymptotic properties:
\begin{itemize}
\item[(i)] If $u \in H^2(\R\times S)$ then  $\norm{\,T_\eps \d_t u - \d_t T_\eps u\,}{L^2} = 
\mathcal{O}(\eps^2)$  ($\eps \to 0$).
\item[(ii)] Let $\al \in \Cinf(\R\times S)$ 
and let $u \in H^1(\R\times S)$ have compact support. Then 
$$
\norm{\,T_\eps (\al u) - \al T_\eps u\,}{L^2} = 
\mathcal{O}(\eps^2) \qquad (\eps \to 0)
$$ 
\end{itemize}
\end{rem}
We have constructed the regularization operators $(T_\eps)_{\eps \in I}$ and the embedding $\iota$ of distributions on $\R\times S$ based on the Riemannian metric $\rho$ given by \eqref{rho}. However, the construction itself does not directly reflect the metric splitting in \eqref{rho} or \eqref{lambda} and so far we have not paid special attention to the foliation by the space-like Cauchy hypersurfaces $\{t\} \times S$ ($t \in \R$). However, this foliation becomes  essential in case of distributions that allow restriction to these hypersurfaces. The latter is true in particular for distributional solutions to the wave equation $\Box_\la u = f$, where $f \in \Cinf(\R, \D'(S))$. In these situations we automatically have $u \in \Cinf(\R, \D'(S))$, since $\Char(\Box_\la)$ does not contain any elements of the form $(t,x;\pm 1, 0) \in T^*(\R\times S)$
(cf.\ \cite{D:93}, 23.65.5). Thus for every $t\in\R$ the value $u(t)$ is an element of $\D'(S)$ and the metric splitting \eqref{lambda} provides us with a Riemannian metric $h_t$ on $S$, which can be used to regularize or embed $u(t)$ according to our general construction. For every $t\in\R$ let $(T^{h_t}_{\eps})_{\eps \in I}$ denote the regularization obtained from the Riemann metric $h_t$ on $S$ and let $\iota_{h_t}$ denote the corresponding embedding $\D'(S) \hookrightarrow \G(S)$. The following statements compare these with the global constructions on $\R\times S$.

\begin{thm}
\noindent (i) The embedding $\iota \col \D'(\R\times S) \hookrightarrow \G(\R\times S)$ respects the metric splitting \eqref{lambda} in the weak sense: If $u\in\Cinf(\R,\D'(S))$ then we have
$$  
   \forall t \in \R:\; \iota(u)(t) \approx \iota_{h_t}(u(t)).
$$

\noindent (ii) If $u \in \Cinf(\R,H^2(S))$, then we have for every compact subset $Z \subseteq \R$
$$
   \sup_{t \in Z} 
   \norm{\,(T_\eps u)(t) - T^{h_t}_{\eps} u(t)\,}{L^2(S)} = 
\mathcal{O}(\eps^2)  \quad (\eps \to 0).
$$
\end{thm}

\begin{proof} (ii) Pick $\chi\in \D(\R)$ such that $\chi\equiv 1$ near $Z$ and let
$\tilde u(t):= \chi(t)u(t)$. It then follows from \cite{DHK:10}, Prop.\ 3.7
that $(T_\eps \tilde u - T_\eps u)_\eps$ is negligible on $\supp(\tilde
u - u)^c \supseteq Z$.
Thus we may without loss of generality assume that the support of $u$ is bounded
in $t$, so $u\in L^2(\R\times S)$. This allows us to employ the integral formulae 
defining $T_\eps$  and $T^{h_t}_\eps$ and obtain
\begin{multline*}
  (T_\eps u)(t) - T^{h_t}_{\eps} u(t) = \\ 
  \frac{1}{2 \pi}
  \int_\R \phi_c(s) \frac{1}{\eps} \FT{F}(\frac{s}{\eps}) 
  \Big( \underbrace{(\cos(s \sqrt{-\Delta_\rho})u)(t)}_{=: w(s,t)} 
  - \underbrace{\cos(s \sqrt{- \Delta_{h_t}}) u(t)}_{=: v(s,t)}
  \Big)\, ds,
\end{multline*}
where $v, w \in \Cinf(\R^2, H^2(S))$ solve the following Cauchy problems, respectively:
\begin{align*}
   (\d_s^2 - \Delta_\rho) w &= 0 
   \quad (\text{in } \D'(\R\times \R\times S)),&
   \quad w \mid_{s=0} &= u,& \quad \d_s w \mid_{s=0} &= 0,\\
   (\d_s^2 - \Delta_{h_t}) v(.,t) &= 0 
   \quad(\text{in } \D'(\R\times S)),&
   \quad v(0,t)  &= u(t),& \quad \d_s v (0,t) &= 0.
\end{align*}
Therefore we have at arbitrary, but fixed $t \in \R$
$$
  \d_s^2 w(.,t) - \Delta_{h_t} w(.,t) = \Delta_\rho w(.,t) - \Delta_{h_t}w(.,t) =: f_t,
$$
where $f_t \in \Cinf(\R, L^2(S))$, since $u$ (hence $w$) has values in $H^2(S)$. We observe that $\psi_t := w(.,t) - v(.,t)$ satisfies the Cauchy problem
$$
  (\d_s^2 - \Delta_{h_t}) \psi_t = f_t 
   \quad(\text{in } \D'(\R\times S)),
   \quad \psi_t(0)  = 0, \quad \d_s \psi_t = 0,
$$ 
which implies
$$
  \psi_t (s) = \int_0^s (s-r) 
  \sinc((s-r) \sqrt{- \Delta_{h_t}}) f_t(r) \, dr.
$$
Hence we obtain
$$
  \norm{\psi_t(s)}{L^2(S)} \leq |\int_0^s (s-r) 
  \norm{f_t(r)}{L^2(S)}\, dr|
  \leq \frac{s^2}{2} \cdot \sup_{|\sig| \leq |s|} 
  \norm{f_t(\sig)}{L^2(S)} 
$$
and therefore arrive at
\begin{multline*}
  \norm{(T_\eps u)(t) - T^{h_t}_{\eps} u(t)}{L^2(S)} \leq
  \frac{1}{2 \pi}
  \int_\R |\phi_c(s)| \frac{1}{\eps} |\FT{F}(\frac{s}{\eps})| 
  \norm{\psi_t(s)}{L^2(S)}\, ds\\
  \leq \frac{\eps^2}{4 \pi}
  \int_{-2c/\eps}^{2c/\eps} |\phi_c(\eps \tau)| |\FT{F}(\tau)| 
  \tau^2 \sup_{|\sig| \leq \eps |\tau|} 
  \norm{f_t(\sig)}{L^2(S)} 
  \, d\tau \\
  \leq \frac{\eps^2}{4 \pi} \norm{\phi_c}{L^\infty} 
  \sup_{|\sig| \leq 2c} \norm{f_t(\sig)}{L^2(S)}
  \norm{\FT{F''}}{L^1} = O(\eps^2) \quad (\eps \to 0)
\end{multline*}
uniformly when $t$ varies in a compact set.

(i) Since association is checked by action on test functions in $\D(S)$ we may reduce to the case that $\supp(u(t))$ is compact. Then $u(t) \in H^l(S)$ for some $l \in \Z$ and since $(1 - \Delta_{h_t})^{\pm l/2}$ commutes weakly with $\iota$ and strongly with $\iota_{h_t}$ the assertion follows from (ii). 
\end{proof}

\newcommand{\SortNoop}[1]{}




\subsection*{Acknowledgment}
G.H.\ thanks the organizers for support by an EPSRC Pathway to Impact Award.
We also acknowledge the support of FWF-projects Y237, P20525, and P23714.
Finally, we would like to thank the referee for several comments that have led to
improvements in the paper.

\end{document}